\documentclass{amsart} 

\usepackage[usenames,dvipsnames]{xcolor}
\usepackage[bookmarks,colorlinks=true,citecolor=OliveGreen,linkcolor=RoyalBlue]{hyperref}
\usepackage{amssymb,amsthm}
\usepackage{mathtools}
\usepackage{mathabx}
\usepackage{xparse}
\usepackage{thmtools}
\usepackage[capitalise]{cleveref}		% Use \cref without environment name
\usepackage{bm}

\usepackage[T1]{fontenc} % Allows diacritics to be copied from the pdf properly
\makeatletter
\def\thmt@refnamewithcomma #1#2#3,#4,#5\@nil{%
  \@xa\def\csname\thmt@envname #1utorefname\endcsname{#3}%
  \ifcsname #2refname\endcsname
    \csname #2refname\expandafter\endcsname\expandafter{\thmt@envname}{#3}{#4}%
  \fi
}
\makeatother
 \declaretheorem[numberwithin=section]{theorem}
\declaretheorem[sibling=theorem]{proposition}

\declaretheorem[sibling=theorem,style=remark]{lemma}

\newcommand{\rest}[1]{\! \upharpoonright_{#1}}

\newcommand{\s}{\sigma}

\renewcommand{\epsilon}{\varepsilon}

\renewcommand{\le}{\leqslant}
\renewcommand{\ge}{\geqslant}
\renewcommand{\leq}{\leqslant}
\renewcommand{\geq}{\geqslant}

\newcommand{\PP}{\mathbb P}

\newcommand{\Cantor}{{2^\omega}}
\newcommand{\bits}{{\{0,1\}}}

\definecolor{greenish}{rgb}{0.0,0.6,0.2}
\definecolor{lightgreenish}{rgb}{0.0,0.8,0.2667}
\definecolor{lightblueish}{rgb}{0.5,0.8,1}

\usepackage{soul}

\title{Dimension~1 sequences are close to randoms}

\author[N.~Greenberg]{Noam Greenberg} 
\address[N.~Greenberg]{School of Mathematics, Statistics and Operations Research, Victoria University of Wellington, Wellington, New Zealand}
\email{greenberg@msor.vuw.ac.nz}
\urladdr{\url{http://homepages.mcs.vuw.ac.nz/~greenberg/}}

\author[J.S.~Miller]{Joseph S. Miller}
\address[J.S.~Miller]{Department of Mathematics, University of Wisconsin--Madison, 480 Lincoln Dr., Madison, WI 53706, USA}
\email{jmiller@math.wisc.edu}
\urladdr{\url{http://www.math.wisc.edu/~jmiller/}}

\author[A.~Shen]{Alexander Shen} 
\address[A.~Shen]{LIRMM, CNRS \& University of Montpellier, 161 rue Ada, 34095, Montpellier, France. Supported by RaCAF ANR grant.}
\email{alexander.shen@lirmm.fr}
\urladdr{\url{http://www.lirmm.fr/~ashen}}

\author[L.B.~Westrick]{Linda Brown Westrick}
\address[L.B.~Westrick]{Department of Mathematics, University of Connecticut, 341 Mansfield Rd U-1009, Storrs, CT 06269, USA}
\email{westrick@uconn.edu}
\urladdr{\url{http://www.math.uconn.edu/~westrick/}}

\thanks{The first and fourth authors were supported by a Rutherford Discovery Fellowship from the Royal Society of NZ. The second author was partially supported by grant \#358043 from the Simons Foundation. The third author was partially supported by RaCAF ANR-15-CE40-0016-01 grant}

\begin{document}
	
\begin{abstract}
We show that a sequence has effective Hausdorff dimension~1 if and only if it is coarsely similar to a Martin-L\"{o}f random sequence. More generally, 
a sequence has effective dimension $s$ if and only if it is coarsely 
similar to a weakly $s$-random sequence.  
Further, for any $s<t$, 
every sequence of effective dimension $s$ can be changed on density at most 
$H^{-1}(t)-H^{-1}(s)$ of its bits to produce
a sequence of effective dimension $t$, and this bound is optimal.
\end{abstract}

\maketitle

%% introduction & advertisement

{The theory of algorithmic randomness defines an individual object in a probability space to be \emph{random}\textbf{} if it looks plausible as an output of a corresponding random process. The first and the most studied definition was given by Martin-L\"of~\cite{Martin_Lof_randomness}: a random object is an object that satisfies all ``effective'' probability laws, i.e., does not belong to any effectively null set. (See~\cite{DowneyHirschfeldtBook,Uspensky2010,Shen2015} for details; we consider only the case of uniform Bernoulli measure on binary sequences, which corresponds to independent tossings of a fair coin.) It was shown by Schnorr and Levin (see~\cite{Schnorr1972STOC,Schnorr:_process_complexity,Levin:73a}) that an equivalent definition can be given in terms of description complexity: a bit sequence {$X\in \Cantor$} is Martin-L\"{o}f (ML) random if and only if the prefix{-free} complexity of its $n$-bit prefix {$X\rest{n}$} is at least $n-O(1)$. (See~\cite{Li.Vitanyi:93,Uspensky2010,Shen2015} for the definition of prefix{-free} complexity and for the proof of this equivalence; one may use also monotone or a priori complexity.)  {This robust class also has an equivalent characterization based on martingales that goes back to Schnorr~\cite{Schnorr1971}.}

The notion of randomness is {in another way} quite fragile: if we take a random sequence and change to zero, say, its $10$th, $100$th, $1000$th, etc.\ bits, the resulting sequence is not random, and for a good reason: a cheater that cheats once in a while is still a cheater. {To consider such sequences as ``approximately random'', one option is to relax the Levin-Schnorr definition by replacing the} $O(1)$ term in the complexity characterisation of randomness by a bigger $o(n)$ term, thus requiring that {$\lim_{n\rightarrow\infty} K(X\rest{n})/n=1$.  Such sequences coincide with the sequences of effective Hausdorff dimension~$1$. (Effective Hausdorff dimension was first explicitly introduced by Lutz~\cite{Lutz:00}. It can be defined in several equivalent ways via complexity, via natural generalizations of effective null sets, and via natural generalizations of martingales; again, see~\cite{DowneyHirschfeldtBook,Uspensky2010,Shen2015} for more information.)

Another approach follows the above example more closely: we could say that a sequence is approximately random if it differs from a random sequence on a set of density $0$. Our starting point is that this also characterizes the sequences of effective Hausdorff dimension $1$.

To set notation, for $n\ge 1$, we let~$d$ be the normalised Hamming distance on~$\bits^n$, the set of binary strings of length $n$: 
\[
	d(\s,\tau) = \frac{\#\left\{  k \,:\,  \s(k)\ne \tau(k) \right\}}{n}; 
\]
and we also denote by~$d$ the Besicovitch distance on Cantor space $\Cantor$ (the space of infinite binary sequences), defined by
\[
	 d(X,Y) = \limsup_{n\to \infty} d(X\rest{n},Y\rest{n}),
\]
where $Z\rest{n}$ stands for the $n$-bit prefix of~$Z$. 
If $d(X, Y)=0$, then we say that $X$ and~$Y$ are \emph{coarsely equivalent}.\footnote{One place this is defined is in \cite{JockuschSchupp:Asymptotic}, where it is called ``generic similarity''.}

\newtheorem*{thm:main}{\cref{thm:main}}
\begin{thm:main}
	A sequence has effective Hausdorff dimension 1 if and only if it is coarsely equivalent to a ML-random sequence. 
	% Let $X\in \Cantor$. Then $\dim(X)=1$ if and only if there is a \textup{ML} random $Y\in \Cantor$ such that $d(X,Y)=0$.
\end{thm:main}

In Section \ref{sec:dims_random}, we generalize this result to sequences of 
effective dimension $s$ in various ways.  Because a sequence $X$ having effective 
dimension $s$ implies that the prefix-free complexity of its 
$n$-bit prefix $X\rest{n}$ is at least $sn-o(n)$, it is natural to consider 
the weakly $s$-randoms, those sequences $X$ such that 
$K(X\rest{n}) \geq sn - O(1)$.

\newtheorem*{thm:dims_srandom}{\cref{thm:dims_srandom}}
\begin{thm:dims_srandom}
Every sequence of effective Hausdorff dimension $s$ is coarsely equivalent to a weakly $s$-random.
\end{thm:dims_srandom}

Along the way to proving this, we pass through the question of how to raise 
the effective dimension of a given sequence while keeping density of 
changes at a minimum. If $d(X,Y)=0$, then $\dim(X)= \dim(Y)$; so sequences of effective Hausdorff dimension $s<1$ cannot be coarsely equivalent to a ML random sequence. It is natural then to ask, what is the minimal distance required between any sequence and a random? By \cref{thm:dims_srandom}, it is equivalent to ask about distances between sequences of dimension~$s$ and dimension~1; and naturally generalising, to ask, for any $0\le s<t\le 1$, about distances between sequences of dimension~$s$ and dimension~$t$. We start with a naive bound. For any $X,Y\in \Cantor$, 
\[
	|\dim(Y) - \dim(X)| \le H(d(X,Y)).
\]
This is our \cref{prop:dim_bound}. Here~$H(p) = -(p\log p + (1-p)\log(1-p))$ is the binary entropy function defined on $[0,1]$. The binary entropy function is used to measure the size of Hamming balls.  
If $V(n,r) = \sum_{k\le nr} \binom{n}{k}$ is the size of a Hamming ball of radius $r<1/2$ in $2^n$, then 
\[
H(r)n - O(\log n) \leq \log(V(n,r)) \leq  H(r)n
\]
(see \cite[Cor.\ 9, p.\ 310]{MacWilliamsSloane:ErrorCorrecting}).

In \cref{prop:dim_bound_achieved}, we will see that this bound is tight, in the sense that if $s<t$ then there are $X,Y\in \Cantor$ with $\dim(X) = s$, $\dim(Y)=t$ and $d(X,Y) = H^{-1}(t-s)$. Note that for~$H^{-1}$ we take the branch which maps $[0,1]$ to $[0,1/2]$.

Bounding the distance from an arbitrary dimension $s$ sequence to the nearest dimension $t$ sequence requires more delicate analysis. For example, fix $0<s<t\le 1$. If $X$ is Bernoulli $H^{-1}(s)$-random, then its dimension is~$s$. But its density of 1s is~$H^{-1}(s)$. If $\dim(Y)\ge t$ then the density of 1s in~$Y$ is at least $H^{-1}(t)$, so $d(X,Y)\ge H^{-1}(t) - H^{-1}(s)$. Note that $H^{-1}(t)-H^{-1}(s) \ge H^{-1}(t-s)$, so this is a sharper bound, and it is tight:

\newtheorem*{thm:dims_dimt}{\cref{thm:dims_dimt}}
\begin{thm:dims_dimt}
For every sequence $X$ with $\dim(X) = s$, and every $t\in (s,1]$, 
there is a $Y$ with $\dim(Y)=t$ and $d(X,Y) \leq H^{-1}(t) - H^{-1}(s)$.
\end{thm:dims_dimt}

In particular, for $t=1$, in light of \cref{thm:main}, we obtain

\newtheorem*{thm:generalization}{\cref{thm:generalization}}
\begin{thm:generalization}
For every $X\in \Cantor$ there is a ML-random sequence $Y$ such that \[
	d(X,Y) \le 1/2 - H^{-1}(\dim (X)).
\]
\end{thm:generalization}

(We however prove \cref{thm:generalization} first, and elaborate on its proof to obtain \cref{thm:dims_srandom} and then \cref{thm:dims_dimt}.)

\medskip

We can also ask, starting 
from an arbitrary random, how close is the nearest
sequence of dimension $s$ guaranteed to be? 
For example, a typical construction of a sequence of 
effective dimension $1/2$ starts with a random and replaces all the even 
bits with $0$.  The distance between the resulting pair is $1/4$, less than 
the $1/2-H^{-1}(1/2) \approx .4$ needed to get to a Bernoulli random, 
but more than 
the $H^{-1}(1-1/2) \approx .1$ lower limit from \cref{prop:dim_bound}.
Here, the latter bound is tight:

\newtheorem*{thm:random_dims}{\cref{thm:random_dims}}
\begin{thm:random_dims}
 For any $Y\in \Cantor$ there is some $X \in \Cantor$ such that $\dim(X)\leq s$ and $d(X,Y)\leq H^{-1}(1-s)$. 
\end{thm:random_dims}

These results mean that in general, the distance from an arbitrary dimension 1 sequence to the nearest dimension $s$ sequence is quite a bit less than the distance from an 
arbitrary dimension $s$ sequence to the nearest dimension 1 sequence. 

\medskip

Finally, we mention that for $t<1$, the bound given by \cref{prop:dim_bound} for the required distance to a sequence of dimension $s<t$ is not optimal; we show below (\cref{prop:naive_bound_not_optimal}) that for the case $t=1/2$ and $s=0$, there are $Y\in \Cantor$ of dimension~$1/2$ with no $X\in \Cantor$ of dimension 0 within distance $H^{-1}(1/2)$. Here decreasing information is not so simple due to the possibility of redundancy of information.

\medskip

Each of these infinitary results have finite versions.  Examples of similar
finite theorems previously appeared in~\cite{BuhrmanEtAl}.  The finite 
versions are proved using either Harper's Theorem, a result of finite 
combinatorics; or estimates on covering Hamming space by balls of a given radius. We adapt those methods, together with some convexity 
arguments, to prove our results.

\medskip

These results exist in the context of a larger set of questions on the 
general theme of asking whether
{every sequence of high effective dimension is obtainable by starting with a random and ``messing it up''.  If a random were messed up only slightly to 
produce a sequence of high effective dimension, it might be possible to 
computably extract a random sequence back out.  It was shown in 
\cite{BienvenuDotyStephan2009,FHPVW2011} 
that if $X$ has positive packing dimension, then $X$ computes a sequence of packing dimension at least $1-\epsilon$ for each $\epsilon>0$.  On the other hand, the second author showed that for any left-c.e. $\alpha \in (0,1)$, there is a sequence of effective dimension $\alpha$ which does not compute any sequence of effective dimension greater than $\alpha$ \cite{Miller:extracting_information},
and in \cite{GreenbergMiller:DNC_and_Hausdorff} it was shown that there is a sequence of dimension~1 that does not compute any random.  Therefore, the 
symmetric differences $X\Delta Y$ which we find here are not, in general, 
computable.  For more references on this type of question, see \cite[Section 13.7]{DowneyHirschfeldtBook}.
}

% Recall that %if 
% $\dim(X) = \liminf_n K(X\rest{n})/n$, {where $\dim(X)$ is the (effective) Hausdorff dimension of a binary sequence $X$ and $K(\s)$ stands for the prefix{-free} complexity of a binary string~$\s$}. It is not hard to see that if $d(X, Y)=0$, then $\dim (X)= \dim(Y)$. Since ML random sequences have dimension~1, if $X$ is coarsely similar to a random, then $X$ has dimension~1. {In the next section, we prove the converse.}

To set notation, for a binary string $\s$ of length~$n$ (we write $\s\in \bits^n$) let 
\[
	\dim(\s) = \frac{K(\s)}{n};
\]
then for an infinite binary sequence $X\in \Cantor$, 
\[
	\dim(X) = \liminf_n \dim (X\rest{n}). 
\]
We can similarly define conditional dimension: 
\[
	\dim(\s \mid \tau) = \frac{K(\s\mid \tau)}{|\s|}.
\]

% \begin{notation}
% 	For an effectively closed set ($\Pi^0_1$ class) $P$, we let $\PP$ be the set of strings~$\s$ extendible on~$P$, that is, $[\s]\cap P \ne \emptyset$. 
% \end{notation}

%%%%%%%%
%%%%%%%%
\section{{Dimension~1 sequences and randoms}}
%%%%%%%%
%%%%%%%%

In this section, we prove \cref{thm:main}. Let $P$ be the set of random sequences with deficiency $0$:
\[
	P = \left\{ Y \,:\,  (\forall n)\, K(Y\rest n) \ge n \right\}.
\]
This $P$ is not empty. {Given a dimension~1 sequence $X$, we will build a $Y\in P$ such that $d(X,Y)=0$.}	

Let~$\PP$ be the set of \emph{extendible strings} of~$P$: the prefixes of elements of~$P$. The following lemma tells us that every string~$\s$ in $\PP$ has many extensions in~$\PP$ of length $2|\s|$. An analogous lemma for supermartingales rather than prefix-free complexity was proved by Merkle and Mihailovic \cite[Rmk.\ 3.1]{MerkleMihailovicJSL}. They gave a clean presentation of the Ku\v{c}era--G\'{a}cs theorem. G\'{a}cs made use of a similar lemma (\cite[Lem.\ 1]{Gacs:KuceraGacs}), which guaranteed a sufficient number of extensions in $\Pi^0_1$ classes.

\begin{lemma} \label{lem:number_of_extensions}
	Every $\s\in \PP$ has at least $2^{|\s| - K(|\s|)-O(1)}$ extensions in~$\PP$ of length $2|\s|$.  
	% There is a constant $k$ such that every $\s\in \PP$ has at least $2^{|\s| - K(|\s|)-k}$  extensions in~$\PP$ of length $2|\s|$.  
\end{lemma}
\begin{proof}
We prove that there is a $k$ such that every $\sigma\in\mathbb{P}$ has at least $2^{2|\sigma|-K(\sigma)-k}$ extensions of length $2|\sigma|$. This is enough, since $K(\sigma)\le |\sigma|+K(|\sigma|)+O(1)$.  %indeed, every string of length $n$ has complexity at most $n+K(n)+O(1)$.

Suppose that some $\sigma\in\mathbb{P}$ has fewer than $2^{2|\s|-K(\s)-k}$ extensions in $\mathbb{P}$ (but has some, otherwise $\sigma$ cannot be in $\mathbb{P}$). Then each extension (denoted by $\sigma'$) has small complexity:
\begin{multline*}
K(\sigma')\le K(\sigma)+K(\sigma'|\sigma,K(\sigma))+O(1)\le\\
\le K(\sigma)+K(\sigma'|k,\sigma,K(\sigma))+O(\log k)\le \\
\le K(\sigma)+2|\sigma|-K(\sigma)-k+O(\log k)=2|\sigma|-k+O(\log k)
\end{multline*}
The first inequality is the formula for the complexity of a pair. In the second one we add $k$ to the condition; the additional $O(\log k)$ term appears. For the third one, if we know $k$, $\sigma$, and $K(\sigma)$, then we can wait until fewer than $2^{2|\sigma|-K(\sigma)-k}$ candidates for $\sigma'$ remain (the set $\mathbb{P}$ is co-c.e.), and then specify each remaining candidate by its ordinal number using a $2|\sigma|-K(\sigma)-k$ bit string; this is a self-delimiting description since its length is known from the condition.

For each such $\sigma'$, we have $K(\sigma')\ge 2|\sigma|$. %-c$
Therefore, $k-O(\log k)\le 0$. %c$. 
So if such a $\sigma$ exists, then $k$ is bounded. Equivalently, if $k$ is sufficiently large, then there is no such $\sigma$, i.e., each $\sigma$ must have at least $2^{2|\s|-K(\s)-k}$ extensions.
\end{proof}

For sets of strings $A,B\subseteq \bits^n$, we let $d(A,B)= \min \left\{ d(\s,\tau) \,:\, \s\in A, \tau\in B  \right\}$. We let $d(\s,A) = d(\{\s\},A)$.

Harper's theorem (\cite{Harper:Theorem}, see also~\cite{FranklFuredi}) says 
that among all subsets $A,B\subseteq \bits^n$ of fixed sizes, a pair with
maximal distance is obtained by taking spheres with opposite centres 
of $0^n$ and $1^n$.  Here a sphere centred at $\s$ is a set~$C$ that (for some $k$) contains the Hamming ball of radius $k/n$ centred at $\s$ and is also contained in the ball of radius $(k+1)/n$ with the same centre.\footnote{{This is, admittedly, an unusual use of ``sphere''; we adopt it from~\cite{FranklFuredi}.}} 

\newtheorem*{thm:harper}{Harper's Theorem}
\begin{thm:harper}
For any sets $A,B\subseteq \bits^n$, there are spheres $\hat A, \hat B$, centred at $0^n$ and~$1^n$ respectively, such that $|A| = |\hat A|$, $|B| = |\hat B|$ and $d(A,B) \le d(\hat A, \hat B)$. 
\end{thm:harper}

A first application, useful for us, is the following.

\begin{lemma} \label{lem:Harper_corollary}
	For every $\epsilon>0$ there is a $q<1$ such that for any~$n$ and any $A\subseteq \bits^n$ of size at least $2^{nq}$, there are at most $2^{nq}$ strings~$\s\in \bits^n$ such that $d(\s,A)> \epsilon$. 
\end{lemma}

\begin{proof}
{For a given} $A\subseteq \bits^n$, let 
	$B = \left\{ \s\in \bits^n \,:\, d(\s,A) > \epsilon  \right\}$.
We need to show that $A$ and $B$ cannot both contain at least $2^{nq}$ elements, for an appropriate choice of $q$. Note that if $|A| \ge 2^{nq}$ and $|B| \ge 2^{nq}$, where $q = H(1/2-\epsilon/2)$, then the inner radii of the spheres $\hat A$ and $\hat B$ from Harper's Theorem are at least $1/2-\epsilon/2 -O(1)/n$, because each sphere is an intermediate set between two balls whose radii differ by $1/n$.  Therefore, $d(A,B) \le \epsilon +O(1)/n$. Note that $H$ is strictly increasing on $[0,1/2]$ and $H(1/2)=1$, so $q<1$.
	
To get rid of the error term $O(1)/n$ that appears because of discretisation, we can decrease $\epsilon$ in advance. Then the statement is true for all sufficiently large $n$. To make it true for all $n$, we choose $q$ so close to $1$ that the statement is vacuous for small $n$; it is guaranteed if $2^{nq}>2^n-1$.
\end{proof}

These tools (Harper's theorem and the entropy bound) were used in~\cite{BuhrmanEtAl} to prove {results on increasing the Kolmogorov complexity of finite strings by flipping a limited number of bits. As an example of this technique, consider the following ``finite version''} of \Cref{thm:main}: for any $\epsilon>0$ there is a $q<1$ such that for sufficiently large~$n$, for any string $\s\in \bits^n$ of dimension at least~$q$ (i.e., $K(\s)\ge nq$), there is a random string $\tau\in \bits^n$ (i.e., $K(\tau)\ge n$) such that $d(\s,\tau)\le \epsilon$. Here is the argument using \cref{lem:Harper_corollary}. The set of random strings has size at least $2^{n-1}$, and is co-c.e.; so once we see that a string~$\s$ is one of the fewer than $2^{qn}$ many strings that are at least $\epsilon$-away from each random string, we can give it a description of length essentially~$nq$. 

A naive plan for the infinite version is to repeat this construction for longer and longer {consecutive blocks of bits of} a given sequence~$X$ of dimension~1, finding closer and closer extensions in a $\Pi^0_1$ class of randoms. 
This fails because the opponent $X$ can copy the extra information that we 
pump into $Y$, erasing our gains.  For example, if $X$ begins with a very 
large string of 0s, we must begin $Y$ with a random string to stay in $P$. 
Then $X$ (which must have dimension 1 eventually) could bring its 
initial segment complexity as close to 1/2 as it likes by copying $Y$ 
from the beginning onto its upcoming even bits.  
Then $Y$ can never take advantage of $X$'s 
complexity to get closer to $X$, 
since that would cause $Y$ to repeat information.  We cannot overcome 
this problem by taking huge steps (so large that $X$ runs out of things to 
copy and must show us new information) because $X$ can still use a similar 
strategy to ensure that the density of symmetric difference is large 
near the beginning of each huge interval, driving up the 
$\limsup_n d(X\rest{n}, Y\rest{n})$ even as we keep $d(X\rest{n},Y\rest{n})$ 
low for $n$ on the interval boundary.

Our solution is to not do an initial segment construction but rather use 
compactness to let ourselves change our mind about our initial segment 
whenever the opponent seems to take advantage of its extra information. 
\cref{lem:main} below shows how to do this.

Let $E$ be the set of all finite sequences $\bar \epsilon = (\epsilon_1,\epsilon_2,\dots, \epsilon_m)$ such that $\epsilon_1 = 1$ and $\epsilon_{k+1}$ equals either $\epsilon_k$ or $\epsilon_{k}/2$ for all $k<m$.  For $m\ge 1$, binary sequences $\s,\tau$ of length $2^m$, and $\bar \epsilon\in E$ of length~$m$, we write
\[
	\s\sim_{\bar \epsilon} \tau
\]
if for all $k\in\{1,\dots, m\}$, 
\[
	d\left(\s\rest{[2^{k-1},2^k)},\tau\rest{[2^{k-1},2^k)}\right) \le \epsilon_k.
\]
%The terms of $s$ are numbered from $0$ to $2^{m-1}$, so 
So we compare the second halves of the strings for $k=m$, the second quarters for $k=m-1$, and so on. (The $0$th bits of $\s$ and $\tau$ are ignored.)

\begin{lemma} \label{lem:main}
	For every $\epsilon>0$ there is an $s<1$ such that for sufficiently large~$m$, for every $\bar\epsilon\in E$ of length~$m$, and for all binary strings $\s,\rho$ of length $2^m$, if
	\begin{enumerate}
		\item $(\bar \epsilon,\epsilon)\in E$ (i.e., $\epsilon \in \{\epsilon_m, \epsilon_m/2\}$),
		\item $\dim (\rho \mid \s) \ge s$,\footnote{Recall that $\dim(\rho\mid \s)= K(\rho\mid \s)/|\rho|$.} and
		\item there is a $\tau\in \PP$ of length $2^m$ such that $\tau\sim_{\bar \epsilon} \s$,
	\end{enumerate}
	then there is a $\nu\in \PP$ of length $2^{m+1}$ such that $\nu\sim_{(\bar \epsilon,\epsilon)}\s\rho$.
\end{lemma}

Note that the guaranteed~$\nu$ need not be an extension of~$\tau$. 

\begin{proof}
Let $n = 2^m$. For a given $\sigma$ and $\bar\epsilon$, let $A$ be the set of all strings $\eta\in \bits^n$ such that for some $\hat \tau \in \PP\cap \bits^n$ we have $\hat \tau\sim_{\bar \epsilon} \s$ and $\hat \tau \eta\in \PP$. The set~$A$ is co-c.e.\ (given~$\s$ and $\bar \epsilon$). Let~$q$ be given by \cref{lem:Harper_corollary} (for $\epsilon$). Now apply \cref{lem:number_of_extensions} to {any} $\hat\tau\in \PP\cap \bits^n$: since $K(n)/n\to 0$ as $n\to\infty$, the size of~$A$ (and even its part that corresponds to this specific $\hat\tau$) is at least $2^{nq}$ (for sufficiently large $m$).

Let $B$ be the set of strings $\pi\in \bits^n$ such that $d(\pi,A)\ge \epsilon$. The set $B$ is c.e., and \cref{lem:Harper_corollary} guarantees that the size of~$B$ is at most $2^{nq}$. This implies that each string in~$B$ can be given a description (conditioned on~$\s$) of length $nq+m+O(1)$ bits;  $m$ bits are used to specify $\bar \epsilon$.
% We note that we are ignoring here, and below, the fact that our codes need to be prefix-free; this adds another $m=\log n$ many bits.
%%% Joe -- But $n$ can be derived from $\sigma$
%; in other words, it can be compressed by $(1-q)n -m$ bits. 
Set $s>q$.  Then since $m = \log n$, 
for sufficiently large~$m$ we have
%each string in~$B$ can be compressed by $n\cdot (1-q)/2$ bits; i.e.{,} 
$\dim(\pi \mid \s) <  s$ for all $\pi\in B$. So
%, letting $s = (1+q)/2$, we get 
$\rho\notin B$. This means that there is some $\eta\in A$ such that $d(\eta,\rho)\le \epsilon$. Let $\hat \tau$ witness that $\eta\in A$. Then $\nu = \hat\tau\eta$ is as required. 
\end{proof}

We finish our preparation with three easy observations.

\begin{lemma} \label{lem:distance}
	Let $X,Y\in \Cantor$ and suppose that 
	\[
		\lim_{m\to \infty} d(X\rest{[2^{m-1},2^m)},Y\rest{[2^{m-1},2^m)}) = 0.
	\]
	Then $d(X,Y)=0$. \qed
\end{lemma}

\begin{lemma} \label{lem:compression}
	Let $X\in \Cantor$ and suppose that $\dim(X)=1$. Then
	\[
		\lim_{m\to \infty} \frac{K(X\rest{[2^m,2^{m+1})} \mid X\rest{2^m})}{2^m} = 1.
	\]
\end{lemma}

\begin{proof}
The complexity of pairs formula shows that 
\[
K(X\rest {2^{m+1}})=
K(X\rest{2^m})+K(X\rest{[2^m, 2^{m+1})}\mid X\rest{2^m})+o(2^m);
\]
the sum can be (almost) maximal only if both terms are (almost) maximal.
\end{proof}

\begin{lemma} \label{lem:converse}
	If $d(X,Y)=0$, then $\dim (X)= \dim(Y)$. \qed
\end{lemma}

\begin{theorem} \label{thm:main}
	Let $X\in \Cantor$. Then $\dim(X)=1$ if and only if there is a \textup{ML} random $Y\in \Cantor$ such that $d(X,Y)=0$.
\end{theorem}

\begin{proof}
One direction is immediate from \Cref{lem:converse}. For the other direction, assume that $\dim (X) =1$. Let
\[
	s_m  = \dim {(X\rest{[2^m,2^{m+1})} \mid X\rest{2^m})}.
\]
Define an infinite sequence $\bar \epsilon = (\epsilon_1,\epsilon_2,\dots)$ such that:
\begin{itemize}
	\item $\epsilon_1 =1$ and $\epsilon_{k+1} \in \{\epsilon_k,\epsilon_k/2\}$ for all $k$,
	\item $\lim_{k\to \infty} \epsilon_k = 0$, and
	\item for all $m$, the triple $\epsilon_m, s_m, m$ satisfies the conclusion of \cref{lem:main} (so if {$s_m$} slowly converges to $1$, we need a sequence $\epsilon_m$ that slowly converges to $0$).
\end{itemize}
We then let 
\[
	Q_m = \left\{ \nu\in \PP\cap \bits^{2^m} \,:\,  \nu \sim_{(\epsilon_1,\dots, \epsilon_m)} X\rest{2^m}  \right\}.
\]
By induction, \cref{lem:main} shows that for all~$m$ the set $Q_m$ is nonempty. Note also that all elements of $Q_m$ have a prefix from $Q_{m-1}$. By compactness, there is a $Y\in P$ such that $Y\rest{2^m}\in Q_m$  for all~$m$. By \Cref{lem:distance}, such~$Y$ is as required.  
\end{proof}

%%%%%%%%%
%%%%%%%%%
\section{{Dimension $s$ sequences and randoms}}\label{sec:dims_random}
%%%%%%%%%
%%%%%%%%%

In this section, we look at the distance between dimension~$s$ sequences and two kinds of randoms. We consider first the density of symmetric difference required to change a sequence of dimension~$s$ into a ML random.
Extending that argument, we then show that every sequence of dimension 
$s$ is coarsely similar to a weakly $s$-random.  
%Finally, we characterize the density that is needed to change a ML random into a sequence of dimension~$s$.

\begin{theorem}\label{thm:generalization}
	For every $X\in \Cantor$ there is a ML-random sequence $Y$ such that \[
	d(X,Y) \le 1/2 - H^{-1}(\dim (X)).
\]
\end{theorem}

In the introduction we saw that
{\Cref{thm:generalization} is optimal by considering the case when $X$ is a 
Bernoulli $H^{-1}(s)$-random sequence.

{The proof of \Cref{thm:generalization} requires several modifications to the work we did in the previous section. For one thing,}
\Cref{lem:distance} fails to generalize, as a positive upper density of $X \Delta Y$ may be greater partway through these intervals than at their boundaries. This could be improved by shortening the intervals. {On the other hand, to make something like \Cref{lem:compression} true for dimension $s$ we would need to increase the size of the intervals; shortening them only makes it} fail even worse. The solution is to go ahead and shorten the intervals, but instead of trying to approach a given target symmetric distance slowly and directly, we let the local symmetric difference rise and fall in accordance with the rise and fall of the conditional complexity of each new chunk.  Then a convexity argument will 
let us conclude that the distance between $X$ and $Y$ is small enough.

Letting the size of the intervals grow quadratically achieves a happy medium. 
In this and all following constructions,
 the $j$th chunk has size $j^2$ and the first $j$ chunks have
concatenated length of $n_j := \sum_{i<j} i^2$, so that $n_j + j^2 = n_{j+1}$.  
There was plenty of freedom in choice of chunk growth rate,
but we must choose something satisfying these conditions: 
We need $n_{j+1}-n_j \ll n_j$ so that the impact of new chunks on
the density of symmetric difference and effective dimension goes to zero in the limit.  And although we could get away with somewhat less,
we also want $j\log j \ll n_{j+1}-n_j$ in order to have room to fit a description of $\overline{\delta}$, which we define next.
 
To replace the set $E$ of descending sequences of $\epsilon_i$, we use the set $D$ of finite sequences $\overline{\delta} = (\delta_1,\dots,\delta_m)$ such that each $\delta_j$ is a fraction of the form $k/j$ for some positive integer $k\leq j$. As before, we write $\sigma \sim_{\overline{\delta}} \tau$ if for each $j$,
 \[d(\sigma\rest{[n_j,n_{j+1})}, \tau\rest{[n_j,n_{j+1})}) < \delta_j.\]

By \cref{thm:main}, it suffices to construct $Y$ of effective dimension 1.  So
 instead of staying inside a tree of randoms, in this and all following 
constructions we stay in trees of the following type.

Given a sequence $\bar t$ of numbers in the unit interval, let 
$$P_{\bar t} = \{Y : (\forall i) \, \dim(Y\rest{[n_i,n_{i+1})} \mid Y\rest{n_i}) \geq t_i\}.$$
Let $\PP_{\bar t}$ denote the set of initial segments of $P_{\bar t}$.
For this first argument, we let $\bar t$ be $\bar 1$, the sequence of all 1s. 
By \cref{lem:compression2} below, if $Y \in P_{\bar 1}$, then $\dim(Y) = 1$.

The following lemma plays the role of \Cref{lem:Harper_corollary} and \Cref{lem:main}.

\begin{lemma}\label{lem:main2}
For every $\epsilon>0$, there is an $N$ such that for every $j\geq N$, 
$\s \in \bits^{n_j}$, $\rho \in \bits^{j^2}$, $\bar \delta \in D$, 
and $\delta = k/j$, if {$\delta > \epsilon$} and
\begin{enumerate}
\item $\dim(\rho\mid\sigma) \geq H(1/2 + \epsilon - \delta)$ and
\item $\sigma \sim_{\overline{\delta}} \tau$ for some $\tau \in \PP_{\bar 1}$,
\end{enumerate}
then there is a $\nu \in \PP_{\bar 1}$ such that $\sigma\rho \sim_{(\overline{\delta}, \delta)} \nu$.
\end{lemma}
\begin{proof}
The sets $A$ and $B$ are defined the same as in the proof of \Cref{lem:main}:
\[
A = \{\eta \in \bits^{j^2} : \exists \hat \tau \in \PP_{\bar 1}\;( \hat \tau \sim_{\bar \delta} \sigma \text{ and } \hat \tau \eta \in \PP_{\bar 1})\},
\]
and $B = \{ \pi \in \bits^{j^2} : d(\pi, A) \geq \delta\}$. 
Now use Harper's Theorem to show that because
$|A| > 2^{j^2-1}$, we have $\log |B| \leq H(1/2-\delta)j^2$.  {Therefore, 
relative to $\sigma$,
codes for elements of $B$ can be given which have length 
$H(1/2 - \delta)j^2 + O(j\log j)$, where $O(j\log j)$ bits 
are used to describe $\bar \delta$.
The difference $H(1/2+\epsilon-\delta) - H(1/2 - \delta)$ 
is bounded away from zero by a fixed amount that 
does not depend on $\delta$.  Therefore, there is an $N$ large enough that 
codes for elements of $B$ are shorter than $H(1/2+\epsilon-\delta)j^2$, 
where the choice of $N$ does not depend on $\delta$.}
\end{proof}

In the construction itself, we define an infinite sequence 
$(\epsilon_1,\epsilon_2,\dots)$ so that $\epsilon_i \rightarrow 0$ 
and $N=i$ witnesses \Cref{lem:main2} for $\varepsilon_i$.  Then we define 
\begin{equation}\label{s_j} 
s_j = \dim{(X\rest{[n_j,n_{j+1})} \mid X\rest{n_j})}
\end{equation}
and finally $\delta_j = 1/2 + \epsilon_j - H^{-1}(s_j) + O(1/j)$ to obtain the sequence 
$(\delta_1,\delta_2,\dots)$ to be used for the definition of $Q_j$ and the proof 
that the $Q_j$ are nonempty.  This produces $Y \in P_{\bar 1}$ with 
$X \sim_{\bar \delta} Y$.  It remains only to examine the relationship 
between the $s_j$, the $\delta_j$, $\dim(X)$, and $d(X,Y)$
for $X \sim_{\overline{\delta}} Y$.  The following variations on 
\Cref{lem:distance} and \Cref{lem:compression} are well-known.

\begin{lemma}\label{lem:distance2}
For $X,Y \in \Cantor$, letting $\delta_i = d(X\rest{[n_i,n_{i+1})}, Y\rest{[n_i,n_{i+1})})$,
\[
d(X,Y) = \limsup_{j\rightarrow\infty} \frac{1}{n_j}\sum_{i=1}^{j-1} \delta_i i^2.
\]
\end{lemma}
\begin{proof}
Since $n_j$ grows slowly enough, both are equal to $\displaystyle\limsup_{j\rightarrow\infty} d(X\rest{n_j}, Y\rest{n_j})$.
\end{proof}

\begin{lemma}\label{lem:compression2}
For $X \in \Cantor$, letting $s_j$ be defined as in (\ref{s_j}),
\[
\dim(X) = \liminf_{j\rightarrow\infty} \frac{1}{n_j}\sum_{i=1}^{j-1} s_i i^2.
\]
\end{lemma}
\begin{proof}
Since $n_j$ grows slowly enough, it suffices to show 
\[
K(\sigma) = \sum_{i=1}^{j-1} s_i i^2 + o(n_j)
\] considering only $\s$ of length $n_j$.  One direction is immediate; the 
other uses $j$ applications of symmetry of information and the fact that 
$j\log j \ll n_j$ (since we condition only on $X\rest{n_j}$ and not its code, 
each application of symmetry of information introduces 
an error of up to $\log n_j \approx \log j$.)
\end{proof}

In our case, with $X\sim_{\overline{\delta}} Y$, we have 
\[
\delta_i = 1/2 + \epsilon_i - H^{-1}(s_i) \geq d(X\rest{[n_i,n_{i+1})}, Y\rest{[n_i,n_{i+1})}).
\]
so 
\[
d(X,Y) \le \limsup_{j\rightarrow\infty} \frac{1}{n_j}\sum_{i=1}^{j-1} \delta_i i^2
=\limsup_{j\rightarrow\infty} \frac{1}{n_j}\sum_{i=1}^{j-1} (1/2-H^{-1}(s_i)) i^2
\]
because $\epsilon_i\rightarrow 0$.  By \Cref{lem:compression2} and the concavity of $1/2-H^{-1}(x)$, this is bounded by $1/2 - H^{-1}(\dim(X))$, as required.  This finishes the proof of \Cref{thm:generalization}.

\bigskip

This method can be extended to answer a question that was suggested to us by
M.\ Soskova.  Recall that a sequence $X$ is called weakly $s$-random 
if it satisfies $K(X\rest{n}) \geq sn -O(1)$.

\begin{theorem}\label{thm:dims_srandom}
Every sequence of effective Hausdorff dimension $s$ is coarsely equivalent to a weakly $s$-random.
\end{theorem}

Simply staying in a tree of $s$-randoms
(direct generalization of \cref{thm:main}) does not provide a strong enough 
lower bound on the number of possible extensions to use Harper's theorem, 
since we now need for more than half of the possible strings of a certain 
length to be available at each stage. 
To ensure this, we instead require our constructed
sequence to buffer its complexity above the level needed to be $s$-random. 
Such a strategy was not a possibility in the 1-random case.  

In the case of finite strings, Harper's Theorem can be used to show,
essentially, that if $x$ is a string of dimension $s$, then by 
changing it on $\epsilon$ fraction of its bits, its dimension can be 
increased to 
$$M(s,\epsilon) := H(\min(1/2, H^{-1}(s) + \epsilon)).$$
We adapt the standard
method to prove the 
following lemma, Case 2 of which is identical to \cref{lem:main2}.

\begin{lemma}\label{lem:main3}
For all $\epsilon$, there is an $N$ such that for all $j\geq N$, all 
$\sigma\in \bits^{n_j}$ and $\rho\in \bits^{j^2}$, 
all $\bar \delta \in D$, $\bar t \in D$ of length $j$, and all $\delta$, if 
\begin{enumerate}
\item $(\bar \delta, \delta) \in D$ and $\delta\geq\varepsilon$ and
\item there is a $\tau \in \PP_{\bar t}$ with $\tau \sim_{\bar \delta} \sigma$,
\end{enumerate}
then,
letting $s = \dim(\rho \mid \sigma)$ and $t = M(s, \delta-\varepsilon) + O(\frac{1}{j})$,
there is a $\nu \in \PP_{(\bar t, t)}$ 
with $\nu \sim_{(\bar \delta, \delta)} \sigma\rho$.
\end{lemma}
\begin{proof}
Like before, let 
\[
A = \{ \eta \in \bits^{j^2} : \exists \hat\tau \sim_{\bar \delta} \sigma \;( 
\hat\tau\eta \in \PP_{(\bar t, t)})\}
\]
and $B = \{\pi \in \bits^{j^2} : d(\pi, A) \geq \delta\}.$

{\bf Case 1.}  Suppose that $0\leq s \leq H(1/2 - \delta - \varepsilon/2)$.  The upper bound is 
chosen so that $t$ is bounded below 1 for $s$ in this interval.  Since $\tau$ exists,
by considering only extensions of it, we can bound $|A| > 2^{j^2} - 2^{tj^2}$.  
Letting $q = M(s, \delta-\varepsilon + \varepsilon/4)$ (note this $q$ is chosen so that 
$t < q < 1$),
for sufficiently large $j$ we have 
\begin{equation}\label{eqn:A_bound}
|A| > 2^{j^2}- V(j^2,H^{-1}(q)),
\end{equation}
 where $V(n,r)$ is the size of a sphere 
of radius $rn$ in $2^n$ (this uses the lower bound for $V(n,r)$ 
mentioned in the introduction). 
How large $j$ has to be for this bound to hold depends on the size of $q-t$,
which in general varies with $s$ and $\delta$. 
But since $q>t$ for all $s,\delta$ with $\varepsilon \leq \delta \leq 1$ and $s$ in the closed 
interval associated to this case, compactness allows us to bound $q-t$ 
away from 0 by a quantity that depends 
only on $\epsilon$.  Assuming $j$ is large enough for (\ref{eqn:A_bound}) to hold,
 Harper's Theorem tells us that 
$$\log |B| \leq H(H^{-1}(q) - \delta)j^2 = H(H^{-1}(s) -\epsilon + \epsilon/4)j^2$$
so $B$ is either empty, or everything in $B$ can be compressed
(relative to $\sigma$) to length 
$H(H^{-1}(s) -3\epsilon/4)j^2 + O(j\log j)$,
(where the $O(j \log j)$ is enough to code the 
parameters $\bar \delta$ and $\bar t$ needed to define $B$ as a 
c.e. set),
and for large enough $j$ the code length is less than $sj^2$.  How large $j$ has to be
depends on the difference $s-H(H^{-1}(s) -3\epsilon/4)$, but 
this can again be bounded in a way that depends only on $\epsilon$.  So 
for sufficiently large $j$, if the conditions 
are satisfied, then $\rho \not\in B$, and the lemma holds.

{\bf Case 2.}  Suppose that $H(1/2 - \delta + \epsilon/2) \leq s \leq 1$.  
Because $\tau$ exists, $|A| \geq 2^{j^2 -1}$, so 
$\log |B| \leq H(1/2-\delta)j^2$ (this is true regardless of the choice of $s$).
Either $B$ is empty, or
this allows the construction of a similar code,
with the needed largeness of $j$ determined by $s-H(1/2-\delta)$, 
which is bounded away from 0 by 
$H(1/2-\delta+\epsilon/2)-H(1/2-\delta)$ for all $s \geq H(1/2 - \delta+ \epsilon/2)$.
For $\delta \in [\varepsilon, 1/2]$, this bound is strictly positive, so by 
compactness there is a uniform lower bound depending only on $\epsilon$. 

Choosing $N$ large enough that $j\geq N$ works for both
cases finishes the proof.
\end{proof}

Now given an $X$ with $\dim(X) = s$, 
define the sequence $s_j$ as in (\ref{s_j}).  Suppose also that an 
infinite sequence $\bar \epsilon \rightarrow 0$ is given, decreasing 
slowly enough that $N=i$ satisfies the previous lemma for $\epsilon_i$.  
Define the infinite sequence $\bar \delta$ by $\delta_i = 2\epsilon_i$, and 
define $\bar t$ by letting $t_j$ be $M(s_j, \delta_j-\epsilon_j) = M(s_j,\epsilon_j)$ rounded up to the nearest rational with denominator $j$. 
The previous lemma together with 
the usual compactness argument provides a sequence $Y\in P_{\bar t}$ 
with $d(X,Y) = 0$ 
and thus $\dim(Y) = s$.  We claim that by 
that by appropriately slow choice of $\bar \epsilon$, we can guarantee 
that $Y$ comes out weakly $s$-random.

The idea is that while $\epsilon$ is held fixed above 0, by changing 
$\epsilon$ fraction of each new chunk, we make $Y$ behave 
like a sequence of dimension strictly greater than $s$ for as long as 
we like.  This allows 
the production and maintenance of a buffer of extra complexity which 
is used to smooth out the bumps in our construction---the logarithmic 
factors from the use of the complexity of pairs formula and the 
possibility for a mid-chunk decrease in the complexity of $Y$.

Specifically, using the complexity of pairs formula repeatedly, and considering 
worst case mid-chunk complexities, we have
$$K(Y\rest{n_j+k}) \geq \sum_{i=0}^j t_i i^2 - O(j\log j)$$ where 
$k<j^2$.  So $Y$ will be $s$-random if we can arrange that eventually
$$\sum_{i=0}^j t_i i^2 - O(j \log j) > s(n_j+j^2),$$
which is what the next lemma guarantees.

\begin{lemma}\label{lem:buffer}
For any constant $c$ and any infinite sequence $\bar s \in D$, let 
$$s = \liminf_{j\rightarrow \infty} \frac{1}{n_j} \sum_{i=0}^j s_i i^2.$$ 
Then there is an infinite $\bar \epsilon \in E$ with 
$\epsilon_i\rightarrow 0$, and a constant $b$, such that for all $j$,
$$\sum_{i=0}^j M(s_i, \epsilon_i) i^2 - cj^2 > sn_j - b.$$
\end{lemma}
\begin{proof}
Observe that for any fixed $\epsilon$, there is a $d>0$ such that 
$$M(x,\epsilon) \geq d + (1-d)x.$$
(On the closed interval 
$[0, H(1/2 - \epsilon)]$, $M(x,\epsilon) > x$, so by compactness, 
on this interval $M(x,\epsilon) > x + d$ for some $d$.  Outside 
this interval, $M(x,\epsilon) = 1$.)  By this bound,
$$\sum_{i=1}^j M(s_i,\epsilon)i^2 \geq dn_j + (1-d)\sum_{i=0}^j s_ii^2.$$
Since $s<1$, there is a $\delta$ such that
$d + (1-d)(s-\delta) > s+\delta$.  So for large enough $N$, we have 
for each $j>N$,
$$\sum_{i=1}^j M(s_i,\epsilon)i^2 \geq dn_j + (1-d)(s-\delta)n_j > (s+\delta)n_j > sn_j + cj^2.$$
All this was done for fixed $\epsilon$, so rename this $N$ as $N_\epsilon$.  
The sequence we want is defined by letting $\epsilon_0=1$, and then let 
$\epsilon_j = \epsilon_{j-1} / 2$ if $j > N_{\epsilon_{j-1}/2}$, and 
otherwise $\epsilon_j = \epsilon_{j-1}$.  The constant $b$ is chosen to 
absorb any irregularity that occurs for $j\leq N_1$.
\end{proof}

Setting $c$ large enough, choosing $\bar \epsilon$ according to 
\cref{lem:buffer}, and constructing $Y$ using $\bar \epsilon$  with 
\cref{lem:main3} produces a $Y$ which is coarsely 
similar to $X$ and with $K(Y\rest{n}) > sn - b$.  This completes the proof 
of \cref{thm:dims_srandom}.

\section{Intermezzo: decreasing dimension}

In the next section, we will generalise \cref{thm:generalization} to increasing dimension from~$s$ to some $t<1$. Now we discuss decreasing dimension. As discussed in the introduction, in one case we can meet the following naive bound.

\begin{proposition}\label{prop:dim_bound}
For any $X,Y\in \Cantor$, 
\[
	|\dim(Y) - \dim(X)| \le H(d(X,Y)).
\]
\end{proposition}
\begin{proof}  Let $s = \dim(X)$ and $\delta = H(d(X,Y))$.
For infinitely many $n$ we have 
$K(Y\rest{n}) \leq (s+\epsilon)n + H(\delta+\epsilon)n + O(\log n),$
where $\epsilon$ is arbitrary.  The first term comes from the inequality 
$K(X\rest{n}) \leq (s+\epsilon)n$ which holds infinitely often. 
The second term comes from a description of the symmetric difference 
$X\rest{n} \Delta Y\rest{n}$, using the fact that eventually 
$d(X\rest{n},Y\rest{n}) < \delta+\epsilon$ to bound the number of possible 
symmetric differences to $2^{nH(\delta+\epsilon)}$.
\end{proof}

  Therefore, if $\dim(Y)=1$ and $\dim(X) = s$, then
$d(X,Y)\geq H^{-1}(1-s)$.  We show that in this case, the bound is tight. This will be based on the finite case:

\begin{lemma}\label{lem:vv}
  For any string $y$ of length $n$ and a given radius $r$, there is a string $x$ in $B(y,r)$
   with $\dim(x) \le 1-H(r) + O(\log n/n)$.
\end{lemma}

\Cref{lem:vv} can be proved using tools from the Vereshchagin-Vit\'anyi theory \cite{VV:10},
  which is surveyed in \cite{VereshchaginShen}.  This is essentially the first part of Theorem 8 of \cite{VereshchaginShen}, modified using the
following facts.  First, the class of Hamming balls satisfies the conditions of Theorem 8 (see \cite[Proposition 28]{VereshchaginShen}).  Second, if $V(n,r)$ is the size of a Hamming ball 
of radius $r$ contained in $2^n$, then $\log(V(n,r)) = H(r)n \pm O(\log n)$ 
(see \cite[Remark 11]{VereshchaginShen}).  
Finally, the complexity of a Hamming ball is within
 $O(\log n)$ of the complexity of its center. 

 Another way to obtain \cref{lem:vv} is by considering covering Hamming space by balls. For any $n$ and $r\in (0,1/2)$, let $\kappa(n,r)$ be the smallest cardinality of a set $C\subseteq \bits^n$ such that $\bits^n = \bigcup_{x\in C} B(x,r)$. Delsarte and Piret (\cite{DelsartePiret:1986}, see \cite[Thm.12.1.2]{CoveringCodes}) showed that 
 \[
 	\kappa(n,r) < 1+n 2^n \ln 2 / V(n,r),
 \]
  where recall from the introduction that $V(n,r)$ is the size of the Hamming space of radius~$r$. As mentioned above, $\log V(n,r)\ge H(r)n - O(\log n)$ (for fixed~$r$), whence
  \[
  	{\log (\kappa(n,r))} \le (1-H(r))n + O(\log n);
  \]
\cref{lem:vv} follows by finding a witness for $\kappa(n,r)$ and giving appropriately short descriptions to all the strings in that witness. 

  Now the counterpart to \cref{prop:dim_bound} follows.

\begin{theorem}\label{thm:random_dims}
For any $Y\in \Cantor$ there is some $X \in \Cantor$ such that $\dim(X)\leq s$ and $d(X,Y)\leq H^{-1}(1-s)$. 
\end{theorem}
Note that if $Y$ is random, then it must be the case that $\dim(X) = s$.

{\begin{proof}
Let $\delta = H^{-1}(1-s)$. 
Given any $Y$, we build $X$ by initial segments.  Split $Y$ into chunks by cutting it at 
the locations $n_j$ 
as before.  By \cref{lem:vv}, for each chunk $y$ from $Y$, 
find a chunk $x$ in $B(y,\delta)$ with 
$\dim(x) \leq s + O(\log n/n)$, where $n = |y|$, and append it to $X$.  
Then $d(X,Y) \le \delta$, and $\dim(X) \le s$, because each chunk satisfies these 
(with $O(\frac{\log n}{n})$ error in the latter case), 
and \cref{lem:distance2} and \cref{lem:compression2} apply.  
\end{proof}}

On the other hand, \cref{prop:dim_bound} is not always optimal. This can be demonstrated with a simple error correcting code.

\begin{proposition} \label{prop:naive_bound_not_optimal}
There is a sequence $Y\in \Cantor$ of dimension $1/2$ such that $\dim(X)>0$ for all $X$ with $d(X,Y)\le H^{-1}(1/2)$. 
\end{proposition}

\begin{proof}
	if $Y$ is the join of a random with itself 
$(t=1/2)$, and if $s=0$, 
suppose $X$ is such that $d(X,Y) \leq H^{-1}(1/2)$.  Then 
$Y\rest{n}$ can be given a description of length
$$K(X\rest{n}) + H^{-1}(1/2)n + n/4 + o(n).$$
Here the description first provides $X\rest{n}$.  Then for each $i$ such that 
$X(2i)\neq X(2i+1)$ (there are at most $H^{-1}(1/2)n$ such $i$), it 
gives $Y(2i)$.  Then it gives a description of 
$\{i : X(2i)=X(2i+1) \neq Y(2i)\}$, a subset of $n/2$ which has size at most 
$\frac{H^{-1}(1/2)}{2}n$, and therefore a description of length $H(H^{-1}(1/2))\frac{n}{2}$.  Since $H^{-1}(1/2) + 1/4 < 1/2$, and for all $n$ we have 
$K(Y\rest{n}) \geq n/2$, we have 
$\dim(X) > 0$.
\end{proof}

In a weaker sense, however, the bound from \cref{prop:dim_bound} is always optimal.

\begin{proposition} \label{prop:dim_bound_achieved}
	For all $s<t$ there are $X,Y\in \Cantor$ with $\dim(X) = s$, $\dim(Y)=t$ and $d(X,Y) = H^{-1}(t-s)$.
\end{proposition}
\begin{proof}
This is similar to the proof of \cref{thm:random_dims}, once we obtain the analogous finite case. Let $r = H^{-1}(t-s)$. Find a witness~$C$ for $\kappa(n,r)$; then find some $D\subseteq C$ of log-size $sn  + O(\log n)$ maximising the size of $S(D) = \bigcup_{x\in D}B(x,r)$.  Because we can guarantee to cover at least $\frac{|D|}{|C|}$ of the space with $S(D)$, we have $\log |S(D)| \ge nt - O(\log n)$. This gives:

\begin{lemma} \label{lem:t_s_finite_case}
	There are at least $2^{nt}n^{O(1)}$ many strings $y\in \bits^n$ which are $H^{-1}(t-s)$-close to a string of dimension at most $s+ O(\log n/n)$. %Therefore, there are $x,y\in \bits^n$ such that $\dim (x)\le s+ O(\log n/n)$, $\dim (x)\ge t - O(\log n/ n)$, and $d(x,y)\le H^{-1}(t-s)$. 
\end{lemma}

Sequences $X$ and~$Y$ as promised by \cref{prop:dim_bound_achieved} are constructed by chunks of size $i^2$, as above. Having defined $X\rest{n_j}$ and $Y\rest{n_j}$, we choose $x_j =  X\rest{[n_j,n_{j+1})}$ and $y_j =  Y\rest{[n_j,n_{j+1})}$ so that $\dim (x_j) \le s+ O(\log n/n)$, $\dim(y_j\mid Y\rest{n_j})\ge t - O(\log n/n)$ (where $n = j^2$), and $d(x_j,y_j)\le r$. 
\end{proof}

\section{Dimension $s$ sequences and dimension $t$ sequences}\label{sec:dims_dimt}

Finally we ask what density of changes are needed to turn a dimension $s$ sequence into 
a dimension $t$ sequence (where $t>s$). By the results of this paper, it is equivalent to
ask for the density of changes needed to turn a weakly 
$s$-random into a weakly $t$-random.  

\Cref{thm:generalization} can 
be generalized. Recall that Bernoulli randoms show that the following bound is optimal.

\begin{theorem}\label{thm:dims_dimt}
For every sequence $X$ with $\dim(X) = s$, and every $t>s$, 
there is a $Y$ with $\dim(Y)=t$ and $d(X,Y) \leq H^{-1}(t) - H^{-1}(s)$.
\end{theorem}

In analogy with the proof of \cref{thm:generalization}, in which the relative complexity of each 
chunk of $X$ was raised to $1$ via whatever distance was necessary, one might first consider 
raising the complexity of each chunk up to $t$.  This fails because of a failure of 
concavity.  Given an individual chunk whose relative complexity $s_i$ is less than $t$,
the density of changes needed to bring it up to $t$ is $\delta_i = H^{-1}(t) - H^{-1}(s_i)$. 
But when $s_i > t$, no changes are needed, so we should choose $\delta_i = 0$.  
However, the resulting function (mapping $s_i$ to $\delta_i$) is not concave.  So a tricky
$X$ could cause this strategy to use distance greater than $H^{-1}(t) - H^{-1}(s)$.  

We use one of two different strategies, with the chosen strategy depending on the particular
$s<t$ pair.  The first strategy is simple: raise the complexity of each chunk as much as 
possible while staying within distance $\delta = H^{-1}(t)-H^{-1}(s)$ of the given chunk.  
This strategy clearly produces a $Y$ with $d(X,Y) \leq \delta$, but showing that 
$\dim(Y)\geq t$ takes a little work and requires the assumption that 
$(1-s)g'(s) \leq (1-t)g'(t)$, where $g(t) = H^{-1}(t)$.

The second strategy is informed by the following reasoning.  If for some $j$, we 
have $\dim(X\rest{n_j}) \approx s$, 
then we should hope to have arranged that $\dim(Y\rest{n_j}) \geq t$, since if we do so,
then we have made the effective dimension of $Y$ large enough.  If we 
have achieved this for $Y\rest{n_j}$ 
and then $X$'s next chunk is relatively random, we can make $Y$'s next 
chunk relatively random for free, so we may as well do so.  This has the effect that 
$(\dim(X\rest{n_{j+1}}), \dim(Y\rest{n_{j+1}}))$ lies on or above 
the line connecting $(s,t)$ to $(1,1)$. 
In this case, our strategy is to use whatever
density of changes are necessary to keep $(\dim(X\rest{n_{j+1}}), \dim(Y\rest{n_{j+1}}))$
on or above this line.  Then it is clear that $\dim(Y) \geq t$, but showing
$d(X,Y) \leq \delta$ requires a little work and the assumption that 
$(1-t)g'(t) \leq (1-s)g'(s)$, complementary to the assumption under which the first strategy
works.

\begin{proof}[Proof of \cref{thm:dims_dimt}]
Given $X$, let $s_i$ be defined as in (\ref{s_j}), and define $\hat s_j$ by 
$$\hat s_j = \frac{1}{n_j}\sum_{i=1}^{j-1} s_ii^2.$$  
For notational convenience, define 
$g(t) = H^{-1}(t)$.

{\bf Case 1.} Suppose that $(1-s)g'(s) \leq (1-t)g'(t)$.  We will construct $Y \in P_{\bar t}$, 
where $\bar t$ is the infinite sequence defined by
$$ t_i = M(s_i, \delta) + O(1/i).$$
(The small extra factor is just to round $t_i$ up to a fraction of the form $k/i$.)
Let $\epsilon_i\rightarrow 0$ 
be an infinite sequence decreasing slowly enough that $N=i$ satisfies 
\cref{lem:main3} for $\epsilon_i$.  Let $\bar \delta$ be defined by $\delta_i = \delta + \epsilon_i$.
By that lemma and the usual compactness argument, there is $Y \in P_{\bar t}$
with $d(X,Y) \leq \delta$.  We must show that if $Y \in P_{\bar t}$, then $\dim(Y) \geq t$.

Let $r(x) = M(x, \delta)$, so that $t_i \geq r(s_i)$.  By \cref{lem:compression2}, 
$$\dim(Y) \geq \liminf_j \frac{1}{n_j}\sum_{i=0}^{j-1} r(s_i)i^2.$$
We would like it if $r(x)$ were convex, so that we could conclude that if $\hat s_j \approx s$, 
then $\frac{1}{n_j} \sum_{i=0}^{j-1} r(s_i)i^2 \geq r(\hat s_j) \approx r(s) = t$.  
But it is not convex, so we use a convex approximation.  Let $\ell(x)$ be the 
tangent line to $r(x)$ at $s$.  Note that since $r$ is increasing, $\ell$ has positive slope. 
Below we will show that $\ell(x) \leq r(x)$ on $[0,1]$.  
Assuming this, we can finish the argument.  Whenever $\hat s_j \leq s$, we have 
$\ell(\hat s_j) = \ell(s) \pm \epsilon = t\pm \epsilon$, with $\epsilon$ vanishing in 
the limit infimum.  On the other hand, if $\hat s_j > s$, then since $\ell$ is increasing 
and each 
$t_i \geq \ell(s_i)$, we have 
$$\frac{1}{n_j} \sum_{i=0}^{j-1} r(s_i)i^2 \geq  \ell(\hat s_j) > \ell(s) = t,$$ as required.

It remains to show that $\ell(x) \leq r(x)$.
The 
proof of the following lemma is elementary but not short; here we state and use it,
but delay a proof sketch to the end of the section.

\begin{lemma}\label{lem:convex}
There is a point $z \in (0,1)$ such that $r(x)$ is convex on $(0,z)$ and concave on $(z,1)$.
\end{lemma}

We can assume $t<1$ (if $t=1$, we are 
covered by \cref{thm:generalization}), so in a neighborhood of $s$, $r(x) = H(g(x) + \delta)$.  
We claim that $r$ is convex at $s$.  
Consider the slope of $r$ at $s$.  We have $r'(x) = H'(g(x)+\delta)g'(x)$, so (using that $H'(g(x)) = 1/g'(x)$)
$$r'(s) = H'(g(t))g'(s) = g'(s)/g'(t).$$
By the case assumption, $r'(s)$ is less than the slope of the line connecting $(s,t)$ to $(1,1)$. 
Since $r'(s)$ is also the slope of $\ell$, it follows that $\ell(1) \leq 1$, a fact we use later. 
If $r$ were concave already at $s$ (and therefore also onwards), its graph would lie
below this line for all $x>s$, so we would have $r(x) <1$ for all $x \in (s,1)$.  But when $g(x)+\delta = 1/2$, 
$r(x)=1$, and this happens for some $x \in (s,1)$.  Therefore, $r(x)$ is convex at $s$, 
so $s\leq z$.

By convexity of $r$ on $(0,z)$, $\ell(x) \leq r(x)$ on $[0,z]$.  
Also, $\ell(1) \leq 1$ by the case assumption,
 and $r(1) = 1$.  Since $\ell(z) \leq r(z)$ and $\ell(1) \leq r(1)$, 
 and since $\ell$ is linear on $[z,1]$ while $r$ is concave on that interval, $\ell(x) \leq r(x)$
  on that interval.  Therefore, $\ell(x) \leq r(x)$ on $[0,1]$, completing the proof of this case.

\medskip

{\bf Case 2.}  Suppose that $(1-t)g'(t) \leq (1-s)g'(s)$.   Let 
$$\ell(x) = \frac{1-t}{1-s}x+ \frac{t-s}{1-s}.$$
This is the line containing $(s,t)$ and $(1,1)$.
We will construct $Y \in P_{\bar t}$, 
where $\bar t$ is the infinite sequence defined by 
$$t_i =  \ell(s_i) + O(1/i).$$
(Again, the extra factor is just to round $t_i$ up to a fraction of the form $k/i$.)  By linearity, 
if each $(s_i, t_i)$ is on or above the graph of $\ell$, then 
$$\left(\frac{1}{n_j}\sum_{i=1}^{j-1} s_ii^2, \frac{1}{n_j}\sum_{i=1}^{j-1} t_ii^2\right)$$
is also.  By \cref{lem:compression2}, the limit infimum of the first coordinate is 
$\dim(X)$, and the limit infimum of the second coordinate is a lower bound for $\dim(Y)$. 
So if $\dim(X) = s$, then $\dim(Y) \geq t$ as required.  Considering also the 
density of changes, we need to 
find $Y \in P_{\bar t}$ with $d(X,Y) \leq g(t) - g(s)$.

Let $\bar \epsilon \rightarrow 0$ be an infinite sequence decreasing slowly enough 
that $N=i$ satisfies \cref{lem:main3} for $\epsilon_i$.  Let $\bar \delta$ be defined by 
$$\delta_i = g(t_i) - g(s_i) + \epsilon_i.$$
Observe that $$M(s_i, \delta_i-\epsilon) + O(1/i) = t_i.$$
This is the complexity increase guaranteed by \cref{lem:main3},
so by that lemma and the usual compactness argument, there is $Y \in P_{\bar t}$ 
with 
$$d(X,Y)  \leq \limsup_j \frac{1}{n_j} \sum_{i=1}^{j-1} \delta_i i^2.$$ 
Now, letting 
\[p(x) = g(\ell(x)) - g(x),\]
 by the uniform continuity of $g$ on $[0,1]$, 
we may additionally assume that $\epsilon_i$ decreases slowly enough that 
$|x-y|<1/i$ implies $|g(x) - g(y)|< \epsilon_i$, so that 
$\delta_i \leq p(s_i) + 2\epsilon_i$.  Since $\epsilon_i\rightarrow 0$, 
$$d(X,Y) \leq  \limsup_j \frac{1}{n_j} \sum_{i=1}^{j-1} p(s_i) i^2.$$

The proof of the concavity of $p$ is elementary but not short.  We just use 
the lemma here and give a sketch of the proof
at the end of this section.
\begin{lemma}\label{lem:concave}
The function $p$ is concave on $[0,1]$.
\end{lemma}
By the concavity of $p$, for each $j$ we have 
$$\frac{1}{n_j} \sum_{i=1}^{j-1} p(s_i) i^2 \leq p(\hat s_j).$$
When $\hat s_j \leq s$, we have $p(\hat s_j) = p(s) \pm \epsilon = \delta\pm\epsilon$,  
with $\epsilon$ 
vanishing in the limit supremum.  
But when $\hat s_j > s$, we still need to bound $p(\hat s_j) \leq \delta = p(s)$.  
Here we use the case assumption. We claim $p$ is decreasing on $[s,1]$, 
because $p'(x) \leq 0$ is true exactly when
$$g'(\ell(x))\frac{1-t}{1-s} - g'(x) \leq 0,$$
which is satisfied when $x = s$ by assumption, and therefore satisfied 
for $x>s$ because $p$ is concave.  Therefore, $p(\hat s_j) \leq \delta$ for $\hat s_j > s$, so 
$d(X,Y) \leq \delta$.
\end{proof}

Now we sketch the lemmas about convexity and concavity used above.  The proofs 
use only undergraduate calculus, mostly of a single variable.

\begin{proof}[Proof sketch of \cref{lem:convex}.]
Given $r(x) = M(x,\delta)$, we need to show there is a $z \in (0,1)$ such that 
$r$ is convex on $(0,z)$ and concave on $(z,1)$.  It suffices to 
consider only what happens on the interval $[0, H(1/2-\delta)]$, 
since $r$ is increasing and 
$r(x) = 1$ for $x\geq H(1/2-\delta)$, continuing the concavity begun at $z$.
For $x$ in this interval,
$$r''(x) = H''(g(x) + \delta)(g'(x))^2 + H'(g(x)+\delta)g''(x)$$
Since these functions all reference $g(x)$, we make the substitution $y = g(x)$.  Note that $y \in [0, 1/2-\delta]$.
With this substitution,
$$r''(x) = H''(y+\delta)/(H'(y))^2 + H'(y+\delta)(-H''(y))/(H'(y))^3$$
Multiplying by $\ln(2)(H'(y))^3$ and dividing by $H''(y)H''(y+\delta)$, $r''(x)$ shares its sign with 
$$w(y) = f(y+\delta) - f(y)$$
where $f(y) = y(1-y)\log_2(1/y-1)$.  When $y = 0$, $w(x) = f(\delta) > 0$ (since $\delta<1/2$).  When $y = 1/2-\delta$, $w(x) = -f(1/2-\delta) < 0$.  (Perhaps neither $r''$ nor $f$ is strictly defined in these places, but the limits exist).  So to show that $r''(x)$ is positive on $(0,z)$ and negative on $(z,H(1/2-\delta))$ for some $z$ in $(0,H(1/2-\delta))$, since $g'(x)$ is positive, it suffices to show that $w$ is strictly decreasing.  Equivalently, $f'(y+\delta) - f'(y)$ is negative; it suffices to show that $f'(y)$ is strictly decreasing on $(0,1/2)$; equivalently $f''(y)$ is strictly negative on this interval. But
$$f''(y) = -(1-2y)/(\ln2(y-y^2)) - 2 \log_2(1/y-1),$$
which is negative, so we are done.
\end{proof}

\begin{proof}[Proof sketch of \cref{lem:concave}.]
Letting $p(x) = g(\ell(x)) - g(x)$ for $\ell$ with slope in $[0,1]$, we show that $p$ is concave.  
Since $\ell(1) = 1$, we write $\ell(x) = ax + 1-a$.  We need to show
$$p''(x) = g''(ax + 1-a)a^2 - g''(x)$$
is non-positive.  When $a=1$, $p''(x) \equiv 0$, so defining $k(a,x) = g''(ax + 1-a)a^2 - g''(x)$, 
it suffices to show that $\frac{\partial k}{\partial a}(a,x)$ is non-negative for $(a,x) \in (0,1]^2$.
We have 
$$\frac{\partial k}{\partial a}(a,x) = g'''(ax+1-a)(x-1)a^2 + 2ag''(ax+1-a),$$
which, with the substitution $y = g(ax+1-a)$, simplifies to 
$$\frac{\partial k}{\partial a}(a,x) = a\left(g'''(H(y))(H(y)-1) + 2g''(H(y))\right).$$
So it suffices to show that this function is non-negative for
all $y \in [g(1-a), 1/2]$.  Expanding the computation of $g'''$ and $g''$, we have 
$$\frac{\partial k}{\partial a}(a,x) = f(y)/(a(H'(y))^5(1-y)^2y^2(\ln 2)^2)$$
where $f(y)$, which has the same sign as $\partial k/\partial a$, is
\begin{multline*} f(y) = \left(3 - (\ln 2)(1-2y) \log_2\left(\frac{1}{y}-1\right)\right) (H(y) -1) \\
+ 2(\ln 2) y (1-y) \left(\log_2\left(\frac{1}{y}-1\right)\right)^2
\end{multline*}
It suffices to show that $f(y)\geq 0$ on $[g(1-a), 1/2]$.  Since $f(1/2) = 0$, it
suffices to show that $f'(y)\leq 0$.  One may check that $f'(1/2) = 0$, so it 
suffices to show that $f''(y)\geq 0$.  We have 
$f''(y) = (h(y)+y(1-y)) / (y^2(1-y)^2)$, where 
$$h(y) = \ln(2-2y) -y(2-3y+2y^2)\ln(1/y-1).$$
It suffices to show that $h(y)\geq 0$.  Since $h(1/2) = 0$, it suffices to show 
that $h'(y)\leq 0$.  One may check $h'(1/2) =0$, so 
it suffices to show that $h''(y)\geq 0$.  We have 
$$h''(y) = \frac{3(1-y)y\ln(1/y-1)+(1-2y)}{2y(1-y)(1-2y)}\geq 0,$$ 
completing the proof.
\end{proof}

\bibliographystyle{alpha}
\bibliography{biblist}

\end{document}